\documentclass[12pt, a4paper, reqno]{amsart}

\usepackage[english]{babel}

\usepackage[T1]{fontenc}
\usepackage{lmodern}

\usepackage{amsmath}
\usepackage{amssymb}
\usepackage{amsthm}

\usepackage{url}

\usepackage{xcolor}
\usepackage{comment}

\usepackage[all]{xy}

\usepackage{hyperref} 

\usepackage{enumitem}
\usepackage[normalem]{ulem}

\usepackage{tikz-cd}
\usetikzlibrary{arrows}

\renewcommand{\epsilon}{\varepsilon}
\renewcommand{\theta}[0]{\vartheta}
\renewcommand{\phi}[0]{\varphi}

\newcommand{\Hc}[0]{\mathcal{H}}

\DeclareMathOperator{\GL}{GL}

\DeclareMathOperator{\End}{End}
\DeclareMathOperator{\Aut}{Aut}

\DeclareMathOperator{\Hol}{Hol}

\DeclareMathOperator{\NHol}{NHol}

\DeclareMathOperator{\res}{res}

\newtheorem{dummy}{Dummy}
\numberwithin{dummy}{section}
\numberwithin{figure}{section}

\newtheorem{theorem}[dummy]{Theorem}

\newtheorem{proposition}[dummy]{Proposition}
\newtheorem{prop}[dummy]{Proposition}

\newtheorem{corollary}[dummy]{Corollary} 

\theoremstyle{definition}
\newtheorem{definition}[dummy]{Definition}
\newtheorem*{notation}{Notation}

\newtheorem{example}[dummy]{Example}

\theoremstyle{remark}

\newtheorem{remark}[dummy]{Remark}

\makeatletter
\def\imod#1{\allowbreak\mkern10mu({\operator@font mod}\,\,#1)}
\makeatother

\numberwithin{equation}{section}

\usepackage{amsthm, mathtools, amssymb, amsfonts, stmaryrd, latexsym,
mathrsfs, todonotes, color}
\allowdisplaybreaks

\begin{document}

\date{1 Dec 2022, 08:30 CEST --- Version 1.03%
}

\title[Finite $p$-groups with a large multiple holomorph]
      {Finite $p$-groups of class two with\\
        a large multiple holomorph}

\author{A.~Caranti}

\address[A.~Caranti]%
 {Dipartimento di Matematica\\
  Universit\`a degli Studi di Trento\\
  via Sommarive 14\\
  I-38123 Trento\\
  Italy\\\endgraf
  ORCiD: 0000-0002-5746-9294} 

\email{andrea.caranti@unitn.it} 

\urladdr{https://caranti.maths.unitn.it/}

\author{Cindy (Sin Yi) Tsang}

\address[Cindy Tsang]%
{Department of Mathematics, Ochanomizu University, 2-1-1 Otsuka, Bunkyo-ku, Tokyo,
Japan
\\\endgraf
  ORCiD: 0000-0003-1240-8102}
\email{tsang.sin.yi@ocha.ac.jp}

\urladdr{http://sites.google.com/site/cindysinyitsang/}



\begin{abstract}
Let $G$ be any group. The quotient group $T(G)$ of the multiple holomorph by the holomorph of $G$ has been investigated for various families of groups $G$. 
In this paper, we shall take $G$ to be a finite $p$-group of class two
for any odd prime $p$, in which case $T(G)$ may be studied using
certain bilinear forms. For any $n\geq 4$, we exhibit examples of $G$
of order $p^{n+{n\choose 2}}$ such that $T(G)$ contains a subgroup
isomorphic to
\[ \GL_n(\mathbb{F}_p) \times \GL_{\binom{n}{2}-n}(\mathbb{F}_p).\]
For finite $p$-groups $G$, the prime factors of the order of $T(G)$ which were known so far all came from $p(p-1)$. Our examples show that the order of $T(G)$ can have other prime factors as well. In fact, we can embed any finite group into $T(G)$ for a suitable choice of $G$.
\end{abstract}

\thanks{The first-named author is a member of GNSAGA--INdAM, Italy.\\
The second-named author is supported by JSPS KAKENHI (Grant-in-Aid for Research Activity Start-up) Grant Number 21K20319.}

\maketitle

\thispagestyle{empty}

\section{Introduction}

Let $G$ be any group, and write $S(G)$ for the group of permutations on the set $G$, where we compose maps from left to right. Consider the right regular representation $\rho : G\rightarrow S(G)$ defined by
\[ x^{\rho(y)} = xy\mbox{ for all }x,y\in G.\]
Let $N_{S(G)}(\cdot)$ denote the normalizer operation in $S(G)$, and put
\[ \Hol(G) = N_{S(G)}(\rho(G)),\,\ \NHol(G) = N_{S(G)}(N_{S(G)}(\rho(G))), \]
which are called the \emph{holomorph} and \emph{multiple holomorph} of $G$, respectively. It is well-known that isomorphic regular subgroups of $S(G)$ are conjugates of each other, and since
\[ N_{S(G)}(\theta^{-1}\rho(G) \theta) = \theta^{-1} N_{S(G)}(\rho(G))\theta = \theta^{-1}\Hol(G)\theta\]
holds for any $\theta \in S(G)$, we easily deduce that the quotient
\[ T(G) = \NHol(G)/\Hol(G)\]
acts regularly on the set
\begin{align*}
 \mathcal{H}(G) & = \{\mbox{regular subgroups $N$ of $S(G)$ which are}\\
 &\hspace{4em}\mbox{isomorphic to $G$ and satisfy $N_{S(G)}(N) = \Hol(G)$}\}.
 \end{align*}
We have the alternative description
\[ \mathcal{H}(G) = \{\mbox{normal regular subgroups of $\Hol(G)$ isomorphic to $G$}\}\]
when $G$ is finite, which is the case of interest in this paper.

The group $T(G)$ was first considered in \cite{Miller} by G. A. Miller, who determined the structure of $T(G)$ for finite abelian groups $G$. Later in  \cite{Mills}
W. H. Mills computed the structure of $T(G)$ for finitely generated abelian groups $G$, which was redone in \cite{fgabelian} by an approach using commutative rings. More recently, T. Kohl \cite{Kohl} determined $T(G)$ for dihedral and generalized quaternion groups $G$. Since his work, the structure of $T(G)$ has attracted more attention and has been determined for other new families of groups $G$ in the past years, such as groups of squarefree order \cite{squarefree}, centerless perfect groups \cite{perfect}, and almost simple groups \cite{ASG}. We remark that \cite{perfect} and \cite{ASG} only treated finite groups, but finiteness may be dropped by~\cite{arXiv:2107.13690}.

It turns out that $T(G)$ is elementary $2$-abelian for all of the families of groups $G$ mentioned above. But there also exist groups $G$ for which $T(G)$ is not elementary $2$-abelian, and most of the examples known so far came from finite $p$-groups with $p$ an odd prime. For example, finite $p$-groups of class at most $p-1$ were considered in \cite{class_two, squarefree}, and finite split metacyclic $p$-groups  in \cite{metacyclic}. It was shown that under suitable hypotheses, the order of $T(G)$ is divisible by $p-1$ or $p$. These examples made us wonder whether the order of $T(G)$ can have prime factors lying outside of $p(p-1)$, which was the motivation of this paper. We are able to show that the answer is affirmative, and in fact $T(G)$ can be made arbitrarily large. Our main result is:

\begin{theorem}\label{thm:main}Let $p$ be any odd prime and let $n\geq 4$ be any integer.

There exists a finite $p$-group $G$ of class two of order $p^{n+{n\choose 2}}$ such that $T(G)$ is a semidirect product of an elementary abelian subgroup of order 
\[p^{{n\choose 2}{n+1\choose 2}}\]
 by a subgroup which is isomorphic to
\begin{equation}\label{semidirect} \mathbb{F}_p^{\left(\binom{n}{2}-n\right)\times n}\rtimes \left( \GL_n(\mathbb{F}_p)\times \GL_{\binom{n}{2}-n}(\mathbb{F}_p)\right),\end{equation}
where 
the semidirect product in \eqref{semidirect} is given by 
\[Q^{(A,M)} = M^{-1}QA\]
for any $Q\in \mathbb{F}_p^{\left(\binom{n}{2}-n\right)\times n},\, A\in \GL_n(\mathbb{F}_p)$, and $M\in \GL_{\binom{n}{2}-n}(\mathbb{F}_p)$.
\end{theorem}

Since every finite group embeds into $\GL_N(\mathbb{F}_p)$ whenever $N$ is large enough, this implies that:

\begin{corollary}Let $p$ be any odd prime and let $H$ be any finite group.

For all sufficiently large integers $n$, there exists a finite $p$-group $G$ of class two of order $p^{n+{n\choose 2}}$ such that $H$ embeds into $T(G)$.
\end{corollary}

\begin{remark}
  Let $G$ be  one of the groups of  Theorem~\ref{thm:main}.  Since every
  pair  of  regular subgroups  in  $\Hc(G)$  normalize each  other,  our
  construction yields  a large  clique in the  normalizing graph  of $G$
  (see~\cite[Section~7]{CS2}).   Equivalently, every  pair chosen  among
  the group  operations on  $G$ associated to  the elements  of $\Hc(G)$
  yields a  bi-skew brace  on $G$~\cite{Childs-bi-skew, mybi},  that is,
  the set of these operations forms a brace block on $G$~\cite{Koc22}.
\end{remark}

Throughout this paper, we shall use the following notation.

\begin{notation}For any group $G$, we write:
\begin{itemize}
\item $\exp(G)=$ the exponent of $G$;
\item $G'=$ the derived subgroup of $G$;
\item $Z(G)=$ the center of $G$;
\item $\mathrm{Frat}(G)=$ the Frattini subgroup of $G$;
\item $\Aut_c(G)=$ the subgroup of $\Aut(G)$ consisting of the automorphisms which induce the identity on $G/Z(G)$;
\item $\Aut_z(G)=$ the subgroup of $\Aut(G)$ consisting of the automorphisms which induce the identity on $Z(G)$.
\end{itemize}
For any $x,y\in G$, we also write 
\[ x^y = y^{-1}xy\mbox{ and }[x,y]=x^{-1}x^y.\]
Let us recall that in any group $G$ of class $2$, we have 
\[ (xy)^d = x^dy^d[y,x]^{d\choose 2}\mbox{ for all }x,y\in G\mbox{ and }d\in \mathbb{N}.\]
We shall use this identity frequently in some of the calculations.

Finally, the symbol $p$ shall always denote an odd prime.
\end{notation}

Here is a brief outline of this paper. In Section \ref{sec:bilinear forms}, we shall describe a technique, due to the first-named author \cite{class_two}, which allows us to study normal regular subgroups $N$ of $\Hol(G)$ via bilinear forms
\[ \Delta : G/Z(G) \times G/G'\rightarrow Z(G)\]
when $G$ is a finite $p$-group of class two. For the purpose of computing $T(G)$, we only want the $N$ that are isomorphic to $G$. In Section \ref{sec:iso class}, we shall discuss the isomorphism class of $N$ in terms of the corresponding bilinear form $\Delta$ when $G' = Z(G)$. This extra assumption allows us to define the notions of symmetric and anti-symmetric on $\Delta$, which make things significantly simpler. In Section \ref{sec:group structure}, we shall study the structure of $T(G)$ in connection with bilinear forms when $\Aut(G)=\Aut_c(G)$. It turns out that $T(G)$ decomposes as a semidirect product
\[T(G) = \mathcal{S}\rtimes \mathcal{S}',\]
where loosely speaking the part $\mathcal{S}$ comes from symmetric forms and $\mathcal{S}'$ comes from anti-symmetric forms. In Section \ref{sec:special}, we shall specialize our findings to a family of groups $G$ constructed in \cite{special_groups}. These groups $G$ satisfy $G ' = Z(G) = \mathrm{Frat}(G)$ so that the study of the bilinear forms $\Delta$ in question reduces to linear algebra. We shall show that the structure of $T(G)$ is given as in Theorem \ref{thm:main} for these groups $G$ under suitable conditions.

\subsection*{Acknowledgements} We thank the referee for helpful suggestions.


\section{Multiple holomorph via bilinear forms}
\label{sec:bilinear forms}

Let us first recall the following result from \cite[Theorem 5.2]{perfect} and 
\cite[Theorem 1.2]{class_two}.

\begin{theorem}\label{thm:gamma} Let $G$ be any finite group.

The following data are equivalent.
\begin{enumerate}[label=\emph{\arabic*.}]
\item A normal regular subgroup $N$ of $\Hol(G)$.
\item An anti-homomorphism $\gamma : G\rightarrow \Aut(G)$ such that 
\[\gamma(x^\beta) = \gamma(x)^\beta\mbox{ for all $x\in G$ and $\beta\in \Aut(G)$}.\]
\item A group operation $\circ$ on $G$ such that $\Aut(G)\leq \Aut(G,\circ)$ and
\[ (xy)\circ z = (x\circ z)\cdot z^{-1}\cdot (y\circ z)\mbox{ for all }x,y,z\in G,\]
where $z^{-1}$ denotes the inverse of $z$ in $G$. Note that the identity of $(G,\circ)$ coincides with that of $G$.
\end{enumerate}
Moreover, these data are related as follows.
\begin{enumerate}[label=\emph{(\roman*)}]
\item $N=\{\gamma(x)\rho(x): x\in G\}$.
\item $x\circ y = x^{\gamma(y)}y$ for all $x,y\in G$.
\item $(G,\circ)\simeq N$ via $x\mapsto \gamma(x)\rho(x)$.
\end{enumerate}
\end{theorem}

\begin{definition}
In the situation of Theorem~\ref{thm:gamma}, we shall refer to the group $(G, \circ)$ as the \emph{circle group}.
\end{definition}

Theorem \ref{thm:gamma} considers all normal regular subgroups of
$\Hol(G)$, but only those which are isomorphic to $G$ correspond to
elements of $T(G)$. The connection is given by the following result of  \cite[Lemma 4.2]{perfect}, which is a rephrasing of a result of~\cite{Miller}.

\begin{prop}\label{prop:theta} Let $G$ be any finite group. Let $N$ be a normal regular subgroup of $\Hol(G)$ with corresponding operation $\circ$ on $G$ such that $N$ is isomorphic to $G$.

For any $\theta\in \NHol(G)$ with $1^\theta = 1$, the following are equivalent.
\begin{enumerate}[label=\emph{(\alph*)}]
\item $\theta : G \rightarrow (G,\circ)$ is a group isomorphism.
\item $\theta$ conjugates $\rho(G)$ to $N$, namely $\rho(G)^\theta = N$.
\end{enumerate}
\end{prop}

For finite $p$-groups of class two, we further have the following linear technique from \cite[Proposition 2.2]{class_two} by the first-named author. We shall employ this approach because many of the calculations reduce to linear algebra and hence are much easier.

\begin{theorem}\label{thm:bilinear forms}Let $G$ be any finite $p$-group of class two.

There is a one-to-one correspondence between the following:
\begin{enumerate}[label=\emph{(\arabic*)}]
\item Normal regular subgroups $N$ of $\Hol(G)$ such that
\begin{equation}\label{gamma cond} \gamma(x) \in \Aut_c(G) \cap \Aut_z(G)\mbox{ for all }x\in G\end{equation}
for the corresponding anti-homomorphism $\gamma : G\rightarrow\Aut(G)$.
\item Bilinear forms $\Delta : G/Z(G)\times G/G'\rightarrow Z(G)$ such that
\begin{equation}\label{eqn:bilinear} \Delta(x^\beta,y^\beta) = \Delta(x,y)^\beta\mbox{ for all }x,y\in G\mbox{ and }\beta\in \Aut(G).\end{equation}
\end{enumerate}
The correspondence is given by
\[ \Delta(x,y) =x^{-1}x^{\gamma(y)}\mbox{ for all }x,y\in G,\]
and in this case the corresponding operation $\circ$ on $G$ is given by
\begin{equation}\label{eqn:circ} 
x\circ y = x^{\gamma(y)}y = x y\Delta(x,y)\mbox{ for all }x,y\in G.
\end{equation}
For simplicity, we are writing $\Delta(x,y)$ instead of $\Delta(xZ(G),yG')$.
\end{theorem}

In what follows, let us assume that $G$ is a finite $p$-group of class two so that Theorem \ref{thm:bilinear forms} applies. For brevity, let us put
\[ B = \{ \mbox{bilinear forms $\Delta : G/Z(G) \times G/G'\rightarrow Z(G)$ satisfying (\ref{eqn:bilinear})}\},\]
and note that $B$ is an abelian group under pointwise multiplication in $Z(G)$. Since $G$ has class two, the commutator $[\, ,\,]$ on $G$ is bilinear and we may use it construct elements of $B$ as follows.

\begin{example}\label{example:power maps}For each $c\in \mathbb{Z}$, we have the power bilinear form
\[\Delta_{[c]}: G/Z(G) \times G/G' \rightarrow Z(G);\,\ \Delta_{[c]}(x,y) = [x,y]^c,\]
which clearly satisfies (\ref{eqn:bilinear}). 
\end{example}

More generally, we can replace the power map $[\, ,\,]\mapsto [\,,\,]^c$ on $G'$ by any endomorphism on $G'$, as it is suggested by the following.

\begin{example}\label{example:sigma}For each $\sigma\in \End(G')$, we have the bilinear form
\[ \Delta_\sigma : G/Z(G) \times G/G'\rightarrow Z(G);\,\ \Delta_\sigma(x,y) = [x,y]^{\sigma},\]
which satisfies (\ref{eqn:bilinear}) if and only if 
\[\sigma\overline{\beta} = \overline{\beta}\sigma \mbox{ for all }\beta\in \Aut(G), \]
where $\overline{\beta}$ denotes the automorphism on $G'$ induced by $\beta$.
\end{example}

Consider $\Delta\in B$ and let $\circ$ denote the corresponding operation on $G$. We know that $\Delta$ gives rise to an element of $T(G)$ if and only if $(G,\circ)$ is isomorphic to $G$. In this case, by Proposition \ref{prop:theta}, the corresponding element of $T(G)$ is the coset $\theta\Hol(G)$, where 
\[\theta : G\rightarrow (G,\circ)\]
denotes any choice of isomorphism, namely a bijection such that
\begin{equation}\label{iso condition} (xy)^\theta = x^\theta \circ y^\theta = x^\theta y^\theta \Delta(x^\theta,y^\theta)\mbox{ for all }x,y\in G.\end{equation}
We end this section with the following example from \cite{class_two}.

\begin{example}\label{example:power maps'}For each $c\in \mathbb{Z}$, consider $\Delta_{[c]}\in B$ and let
\[ x\circ_{[c]} y = xy\Delta_{[c]}(x,y) = xy[x,y]^{c}\]
denote the corresponding group operation on $G$. In the case that
\[c\not\equiv\textstyle-\frac{1}{2} \hspace{-3mm}\pmod{p},\mbox{ in other words } 2c+1\not\equiv 0\hspace{-3mm}\pmod{p},\]
we know from \cite{class_two} that $(G,\circ_{[c]})$ is isomorphic to $G$. Explicitly, we may obtain an isomorphism via the power map
\[ \theta_d : G\rightarrow (G,\circ_{[c]});\,\ x^{\theta_d} = x^d,\]
where $d\in \mathbb{Z}$ is any integer such that
\[  d(2c+1) \equiv 1\hspace{-3mm}\pmod{\exp(G')},\mbox{ that is }\textstyle c\equiv -\frac{d-1}{2d}\hspace{-3mm}\pmod{\exp(G')}.\]
The bijectivity of $\theta_d$ is clear. For any $x,y\in G$, we have
\begin{align*}
(xy)^{\theta_d}& = (xy)^d\\
& = x^dy^d[y,x]^{\frac{d(d-1)}{2}}\\
& = x^dy^d[x^d,y^d]^{-\frac{d-1}{2d}}\\
& = x^{\theta_d}y^{\theta_d}\Delta_{[c]}(x^{\theta_d},y^{\theta_d}),
\end{align*}
where the second equality holds because $G$ has class two. We then see from (\ref{iso condition}) that $\theta_d$ is indeed an isomorphism. The set 
\[ \{\theta_d\Hol(G) : d\in \mathbb{Z}\mbox{ coprime to $p$}\},\]
consisting of the cosets defined by these $\theta_d$ is precisely the cyclic subgroup of $T(G)$ of order $(p-1)p^{r-1}$ constructed in \cite[Proposition 3.1]{class_two}, where $p^r = \exp(G/Z(G))$. We note that $\exp(G/Z(G)) = \exp(G')$ because $G$ is assumed to have class two.
\end{example}


\section{Isomorphism class of the circle group}
\label{sec:iso class}

Throughout this section, we shall assume that $G$ is a finite $p$-group of class two such that $G' = Z(G)$. We are then looking at the set
\[ B = \{ \mbox{bilinear forms $\Delta : G/G' \times G/G'\rightarrow G'$ satisfying (\ref{eqn:bilinear})}\},\]
which is an abelian group under pointwise multiplication in $G'$. Since both of the arguments come from $G/G'$, the notions of symmetric and anti-symmetric are defined. In particular, we have the subgroup
\[ S = \{\Delta \in B : \Delta(y,x) = \Delta(x,y)\mbox{ for all }x,y\in G\}\]
consisting of the symmetric forms, and similarly the subgroup
\[ S' = \{\Delta\in B : \Delta(y,x) = \Delta(x,y)^{-1}\mbox{ for all }x,y\in G\}\]
consisting of the anti-symmetric forms. It is clear that $S$ and $S'$ intersect trivially. Since every $\Delta\in B$ may be decomposed as
\[ \Delta(x,y) =\left(\Delta(x,y)\Delta(y,x)\right)^{1/2} \cdot\left(\Delta(x,y) \Delta(y,x)^{-1}\right)^{ 1/2},\]
we see that $B=S\times S'$ is a direct product of $S$ and $S'$.

Given any $\Delta\in B$ with corresponding operation $\circ$ on $G$, we wish to determine when $(G,\circ)$ is isomorphic to $G$. By the next proposition and corollary, we only need to consider anti-symmetric forms.

\begin{prop}\label{slice in B}Let $\Delta_1,\Delta_2\in B$ be such that $\Delta_1^{-1}\Delta_2$ is symmetric, and let $\circ_1,\circ_2$, respectively, denote their corresponding operations on $G$ as given by (\ref{eqn:circ}).

The groups $(G,\circ_1)$ and $(G,\circ_2)$ are isomorphic.
\end{prop}
\begin{proof} Put $\Delta=\Delta_1^{-1}\Delta_2$, which is symmetric by assumption. Define 
\[\theta : (G,\circ_1)\rightarrow (G,\circ_2);\,\ x^\theta= x\Delta(x,x)^{1/2}.\]
For any $x,y\in G$, we have
\begin{align*}
(x\circ_1 y)^\theta & = (xy\Delta_1(x,y))^\theta \\
 & = xy\Delta_1(x,y)\Delta(xy,xy)^{1/2}\\
 &=xy \Delta_1(x,y)\Delta(x,x)^{1/2}\Delta(y,y)^{1/2}\Delta(x,y)\\
 & = x\Delta(x,x)^{1/2} y\Delta(y,y)^{1/2} \Delta_2(x,y)\\
 &= x\Delta(x,x)^{1/2}\circ_2 y\Delta(y,y)^{1/2}\\
&= x^\theta\circ_2 y^\theta.
\end{align*}
It follows that $\theta$ is a homomorphism.  Since any $x\in \ker(\theta)$ must lie in $G'$ and $\theta$ is the identity on $G'$, we deduce that $\theta$ is injective, and hence bijective because $G$ is finite. This proves that $\theta$ is an isomorphism.
\end{proof}

Taking $\Delta_1$ to be the trivial bilinear form, whose corresponding operation $\circ$ coincides with that of $G$, in particular we have:

\begin{corollary}\label{sym cor}Let $\Delta\in S$ and let $\circ$ denote the corresponding operation on $G$ as given by (\ref{eqn:circ}).

The groups $G$ and $(G,\circ)$ are isomorphic.
\end{corollary}

In contrast to symmetric forms, the anti-symmetric forms are much harder to understand, except those considered in Example \ref{example:power maps}. Below, for the sake of completeness,  let us just prove several facts that can be established in general.

Let $\Delta\in B$ and let $\circ$ denote the corresponding operation on $G$. For any $x\in G$, let us write $x^{\ominus 1}$  for its inverse in the group $(G,\circ$). A simple calculation shows that 
\[ x^{\ominus1} = x^{-1}\Delta(x,x).\]
For any $x,y\in G$, their commutator in $(G,\circ)$ is then given by
\begin{align}\notag
[x,y]_\circ & = x^{\ominus1}\circ y^{\ominus1}\circ x\circ y \\\notag
& = (x^{-1}\Delta(x,x)\circ y^{-1}\Delta(y,y))\circ (x\circ y)\\\notag
& = x^{-1}y^{-1}\Delta(x,x)\Delta(y,y)\Delta(x^{-1},y^{-1})\circ xy\Delta(x,y)\\\notag
& = [x,y]\Delta(x,x)\Delta(y,y)\Delta(x,y)^2\Delta(x^{-1}y^{-1},xy)\\\label{commutator}
& = [x,y]\Delta(x,y)\Delta(y,x)^{-1}.
\end{align}
This implies that we have the inclusions
\begin{equation}\label{inclusions}
 (G,\circ) ' \leq G' = Z(G) \leq Z(G,\circ).
 \end{equation}
 Hence, the group $(G,\circ)$ is either abelian or has class two, and for it to be isomorphic to $G$, the inclusions in (\ref{inclusions}) must both be equalities.

 \begin{prop}\label{iso prop}Let $\Delta\in S'$ and let $\circ$ denote the corresponding operation on $G$ as given in (\ref{eqn:circ}).
\begin{enumerate}[label=\emph{(\alph*)}]
\item $(G,\circ)$ is abelian if and only if $\Delta = \Delta_{[-1/2]}$.
\item $(G,\circ)'=G'$ if and only if the image of $\Delta_{[1/2]}\Delta$ generates $G'$.
\item $Z(G) = Z(G,\circ)$ if and only if $\Delta_{[1/2]}\Delta$ is non-degenerate, that is, we have $(\Delta_{[1/2]}\Delta)(x,y) = 1$ for all $y\in G$ only when $x\in G'$.
\end{enumerate}
\end{prop}
\begin{proof}Since $\Delta$ is anti-symmetric, the commutator (\ref{commutator}) becomes
\[ [x,y]_\circ = [x,y]\Delta(x,y)^2,\mbox{ or equivalently }[x,y]_\circ^{1/2} = (\Delta_{[1/2]}\Delta)(x,y).\]
Since $(G,\circ)$ and $G$ have the same identity element, it follows that
\begin{align*}
(G,\circ)\mbox{ is abelian}&\iff \forall x,y\in G : [x,y]_\circ = 1\\
& \iff \forall x,y\in G:  [x,y]\Delta(x,y)^2 = 1\\
& \iff \forall x,y\in G: \Delta(x,y) =[x,y]^{-1/2},
\end{align*}
which proves (a). Observe that the operation $\circ$ coincides with that of $G$ inside $G'$. From the inclusion $(G,\circ)'\leq G'$, we then see that
\begin{align*}
(G,\circ)' = G' &\iff [x,y]_\circ\mbox{ with $x,y\in G$ generate }G'\\
&\iff [x,y]^{1/2}_\circ\mbox{ with $x,y\in G$ generate }G'\\
&\iff (\Delta_{[1/2]}\Delta)(x,y)\mbox{ with $x,y\in G$ generate }G',
\end{align*}
and this yields (b). Finally, notice that $Z(G,\circ)$ is the set consisting of all of the elements $x\in G$ for which
\[ [x,y]_\circ =1,\mbox{ or equivalently }(\Delta_{[1/2]}\Delta)(x,y)=1\]
holds for all $y\in G$. Since $Z(G)\leq Z(G,\circ)$ and $G'=Z(G)$, we deduce that (c) indeed holds.
\end{proof}

\begin{example}\label{example:sigma'}For each $\sigma\in\End(G')$, consider $\Delta_\sigma$ from Example \ref{example:sigma}, and assuming that it does satisfy (\ref{eqn:bilinear}), let
\[ x\circ_\sigma y = xy\Delta_\sigma(x,y) = xy[x,y]^{\sigma}\]
denote the corresponding operation on $G$. Note that
\[ ( \Delta_{[1/2]}\Delta)(x,y) = [x,y]^{1/2+\sigma} = [x,y]^{\frac{1}{2}(1+2\sigma)}.\]
We then deduce from Proposition \ref{iso prop} the following:
\begin{enumerate}[label=(\alph*)]
\item If $1+2\sigma\not\in \Aut(G')$, then $(G,\circ_\sigma)'$ is a proper subgroup of $G'$, and so $(G,\circ_\sigma)$ cannot be isomorphic to $G$.
\item If $1+2\sigma\in \Aut(G')$, then $(G,\circ_\sigma)'= G'$ holds, and
\begin{align*}
\forall y \in G:  ( \Delta_{[1/2]}\Delta)(x,y)= 1
&\implies \forall y \in G:  [x,y]^{\frac{1}{2}(1+2\sigma)} = 1\\
&\implies \forall y\in G: [x,y]=1\\
& \implies x \in Z(G) = G',
\end{align*}
whence $Z(G)=Z(G,\circ)$ holds as well.
\end{enumerate}
However, even though we have equalities
\begin{equation}\label{sigma eqn}(G,\circ_\sigma)' = G' = Z(G) = Z(G,\circ_\sigma),\end{equation}
it is unclear at this stage for which $\sigma$ we have that $(G,\circ_\sigma)$ and $G$ are actually isomorphic. The only thing that can be said in general is that they are always isoclinic.
\end{example}

Recall that two groups $\Gamma_1,\Gamma_2$ are \emph{isoclinic} if there are isomorphisms 
\[ \varphi : \Gamma_1/Z(\Gamma_1) \rightarrow \Gamma_2/Z(\Gamma_2),\,\ \psi : \Gamma_1'\rightarrow \Gamma_2'\]
which are compatible with commutators, namely
\[   [\gamma^\varphi,\delta^\varphi] = [\gamma,\delta]^{\psi} \mbox{ for all }\gamma,\delta\in \Gamma_1.\]
We end this section with the next proposition.

\begin{prop}Let $\sigma\in \End(G')$ be such that $1+2\sigma\in \Aut(G')$ and also assume that $\Delta_\sigma$ satisfies (\ref{eqn:bilinear}).

The groups $G$ and $(G,\circ_\sigma)$ are isoclinic.
\end{prop}
\begin{proof}We know that the equalities in (\ref{sigma eqn}) hold. Take
\[ \varphi : G/Z(G) \rightarrow (G,\circ)/Z(G,\circ),\,\ \psi : G'\rightarrow (G,\circ)',\]
respectively, to be the identity map and the automorphism $1+2\sigma$. It is clear from (\ref{eqn:circ}) that the operation $\circ$ coincides with that of $G$ when taken modulo $Z(G)$ as well as when restricted to $G'$. Thus, both $\varphi$ and $\psi$ are isomorphisms. For any $x,y\in G$, it follows from (\ref{commutator}) that
\[ [x^\varphi,y^\varphi]_\circ  = [x,y]_\circ = [x,y]^{1+2\sigma} = [x,y]^\psi,\]
and this proves the claim.
\end{proof}

\section{Group structure of the multiple holomorph}\label{sec:group structure} 

Throughout this section, we shall assume that $G$ is a finite $p$-group of class two such that $G'=Z(G)$ and $\Aut(G) =\Aut_c(G)$. These two conditions imply that $\Aut(G)=\Aut_z(G)$ also, and so every normal regular subgroup $N$ of $\Hol(G)$ satisfies (\ref{gamma cond}). Moreover, the condition (\ref{eqn:bilinear}) is now vacuous, so we are simply looking at the set
\[ B = \{\mbox{bilinear forms $\Delta : G/G'\times G/G'\rightarrow G'$}\}.\]
As explained in Section \ref{sec:bilinear forms}, each $\Delta \in B$, with corresponding operation $\circ$ say, gives rise to an element $\theta\Hol(G)$ of $T(G)$ whenever there is an isomorphism $\theta:G\rightarrow (G,\circ)$. By Theorem \ref{thm:bilinear forms}, every element of $T(G)$ arises in this way, and the associated bilinear form is unique. 

Now, consider an arbitrary element of $T(G)$ represented as
\[ \theta\Hol(G),\mbox{ where $\theta : G\rightarrow (G,\circ)$ is an isomorphism}\]
for the operation $\circ$ corresponding to some $\Delta\in B$. The isomorphism $\theta$ induces via restriction isomorphisms
\[ G/G' \rightarrow (G,\circ)/(G,\circ)',\,\ G'\rightarrow (G,\circ)'.\]
But from (\ref{inclusions}), we must have the equalities
\[ (G,\circ)' = G' = Z(G) = Z(G,\circ).\]
Also from (\ref{eqn:circ}), the operation $\circ$ coincides with that of $G$ when taken modulo $G'$ and when restricted to $G'$. Thus $\theta$ induces automorphisms 
\begin{equation}\label{res} \res_c(\theta) : G/G'\rightarrow G/G',\,\ \res_z(\theta) : G'\rightarrow G'.\end{equation}
Note that they do not depend on the choice of $\theta$, for if $\theta' : G\rightarrow (G,\circ)$ is another isomorphism, then $\theta'\theta^{-1}\in \Aut(G)$, and by assumption
\[ \Aut(G) = \Aut_c(G) = \Aut_z(G).\]
This gives us a well-defined map
\[ \res : T(G) \rightarrow \Aut(G/G') \times \Aut(G');\,\ \res = (\res_c,\res_z),\]
which is clearly a homomorphism. 

Observe that the range of $\res$ acts naturally on $B$ via
\[\Delta^{(\alpha,\beta)}(x,y)= \Delta(x^{\alpha^{-1}},y^{\alpha^{-1}})^{\beta} \mbox{ for }\Delta\in B,\, \alpha\in \Aut(G/G'),\, \beta\in \Aut(G').\]
The next proposition shows that the group operations of $T(G)$ and $B$ are related by this action.

\begin{proposition}\label{prop:compare}
For $i=1,2$, let $\Delta_i\in B$ with corresponding operation $\circ_i$ be such that there exists an isomorphism $\theta_i:G\rightarrow (G,\circ_i)$.

The bijection $\theta_1\theta_2:G\rightarrow (G,\circ)$ is an isomorphism for the operation $\circ$ corresponding to the bilinear form $\Delta_1^{\res(\theta_2)}\Delta_2$.
\end{proposition}
\begin{proof}
Put $\Delta =\Delta_1^{\res(\theta_2)}\Delta_2$, which is explicitly given by
\[ \Delta(x,y) = \Delta_1(x^{\theta_2^{-1}},y^{\theta_2^{-1}})^{\theta_2}\Delta_2(x,y) \mbox{ for all }x,y\in G.\]
Clearly $\theta_1\theta_2$ is a bijection, and we shall use the criterion (\ref{iso condition}) to check that $\theta_1\theta_2$ is a homomorphism. For any $x,y\in G$, we have
\begin{align*}
(xy)^{\theta_1\theta_2} & =(x^{\theta_1}y^{\theta_1}\Delta_1(x^{\theta_1},y^{\theta_1}))^{\theta_2}\\
& = (x^{\theta_1}y^{\theta_1})^{\theta_2} \Delta_1(x^{\theta_1},y^{\theta_1})^{\theta_2}\\
& = x^{\theta_1\theta_2}y^{\theta_1\theta_2}\Delta_2(x^{\theta_1\theta_2},y^{\theta_1\theta_2})\Delta_1(x^{\theta_1},y^{\theta_1})^{\theta_2}\\
& =  x^{\theta_1\theta_2}y^{\theta_1\theta_2}\Delta( x^{\theta_1\theta_2}, y^{\theta_1\theta_2}),
\end{align*}
and the claim now follows.
\end{proof}

Recall that $B = S\times S'$ decomposes as a direct product of the subgroups of symmetric and anti-symmetric forms. Let $\mathcal{S}$ and $\mathcal{S}'$ denote the subsets of $T(G)$ consisting of all the elements $\theta\Hol(G)$ which arise from symmetric and anti-symmetric forms, respectively. In fact, both $\mathcal{S}$ and $\mathcal{S}'$ are subgroups of $T(G)$ in view of Proposition \ref{prop:compare}.

\begin{prop}\label{prop ker}We have $\mathcal{S} = \ker(\res)$. \end{prop}
\begin{proof} Let $\theta\Hol(G) \in T(G)$, where $\theta:G\rightarrow (G,\circ)$ is any isomorphism for the operation $\circ$ defined by the corresponding $\Delta\in B$. 

If $\Delta$ is symmetric, then we may take $\theta$ to be
\[x^\theta = x\Delta(x,x)^{1/2}\mbox{ for all }x\in G\]
by the proof of Proposition \ref{slice in B}, and clearly $\theta\Hol(G)\in \ker(\res)$.

If $\theta\Hol(G)\in\ker(\res)$, then
\[ [x,y]= [x,y]^\theta = [x^\theta,y^\theta]_\circ = [x,y]_\circ\mbox{ for all }x,y\in G,\]
and thus $\Delta$ is symmetric by (\ref{commutator}). 
 \end{proof}

Although $B = S\times S'$ decomposes as a direct product, it is not true in general that $T(G) = \mathcal{S}\times \mathcal{S}'$. But we do have a semidirect product.
 
\begin{proposition}\label{prop:semidirect} We have the semidirect product decomposition
\[ T(G) = \mathcal{S}\rtimes \mathcal{S}'.\]
\end{proposition}
\begin{proof}
Let $\theta\Hol(G) \in T(G)$, where $\theta:G\rightarrow (G,\circ)$ is any isomorphism for the operation $\circ$ defined by the corresponding $\Delta\in B$. By Theorem \ref{thm:bilinear forms}, the choice of $\Delta\in B$ is unique. 

If $\theta\Hol(G)\in \mathcal{S}\cap\mathcal{S}'$, then $\Delta$ must be the trivial bilinear form, and $\theta\in \Aut(G)$ because $\circ$ coincides with the operation of $G$ in this case.

The above shows that $\mathcal{S}\cap \mathcal{S}' = 1$. Since $\mathcal{S}$ is a normal subgroup of $T(G)$ by  Proposition \ref{prop ker}, it remains to show that $T(G) = \mathcal{S}\mathcal{S}'$.

Since $B = S\times S'$, we may decompose
\[\Delta = \Delta_0^{-1}\Delta_2,\mbox{ where }\Delta_0\in S\mbox{ and }\Delta_2\in S'.\]
Let $\circ_2$ denote the operation corresponding to $\Delta_2$. Since $\Delta^{-1}\Delta_2 =\Delta_0$ is symmetric, by the proof of Proposition \ref{slice in B}, we know that
\[ (G,\circ)\rightarrow (G,\circ_2);\,\ x\mapsto x\Delta_0(x,x)^{1/2}\]
is an isomorphism. Composing with $\theta$, this implies that
\[ \theta_2 : G\rightarrow (G,\circ_2);\,\ x^{\theta_2} = x^\theta\Delta_0(x^\theta,x^\theta)^{1/2}\]
is an isomorphism. Put
\[ \Delta_1 = \Delta_0^{-\res_1(\theta_2)^{-1}},\mbox{ which clearly lies in }S,\]
and by Proposition \ref{slice in B}, we know that
\[ \theta_1 : G\rightarrow (G,\circ_1); \,\ x^{\theta_1} = x\Delta_1(x,x)^{1/2} \]
is an isomorphism for the operation $\circ_1$ corresponding to $\Delta_1$. We have
\[ \theta\Hol(G) = \theta_1\theta_2\Hol(G) =  \theta_1\Hol(G) \cdot \theta_2\Hol(G)\in \mathcal{S}\mathcal{S}'\]
by Proposition \ref{prop:compare} because 
\[ \Delta = \Delta_0^{-1}\Delta_2 = \Delta_1^{\res(\theta_2)}\Delta_2,\]
and this completes the proof.
\end{proof}

Let us compute the conjugation action of $\mathcal{S}'$ on $\mathcal{S}$. Consider $\Delta\in S$ and $\Delta'\in S'$, with corresponding operations $\circ$ and $\circ'$ respectively, such that there exist isomorphisms
\[ \theta:G\rightarrow (G,\circ)\mbox{ and }\theta':G\rightarrow (G,\circ').\]
By Proposition \ref{prop:compare}, the bilinear forms corresponding to
\[ \theta'^{-1}\Hol(G)\mbox{ and }\theta\theta'\Hol(G),\]
respectively, are equal to
\[(\Delta')^{-\res(\theta'^{-1})}\mbox{ and }\Delta^{\res(\theta')}\Delta'.\]
It follows that the bilinear form corresponding to
\[ \theta'^{-1}\theta \theta'\Hol(G) = \theta'^{-1}\Hol(G) \cdot \theta \theta'\Hol(G)\]
is in turn given by
\[ (\Delta')^{-\res(\theta'^{-1})\res(\theta\theta')}\cdot  (\Delta^{\res(\theta')}\Delta') = \Delta^{\res(\theta')},\]
where the equality holds because $\res(\theta)=1$ by Proposition \ref{prop ker} and $B$ is an abelian group. This means that the action of $\mathcal{S}'$ on $\mathcal{S}$ agrees with that of $\res(\mathcal{S}')$ on $S$. Summarizing, we obtain the following theorem:

\begin{theorem}\label{thm:T(G)} We have an isomorphism
\[ T(G) \simeq S\rtimes \res(\mathcal{S}'),\]
\end{theorem}

\begin{proof}Recall from Corollary \ref{sym cor} that every symmetric form gives rise to an element of $T(G)$. Since $\mathcal{S} = \ker(\res)$ by Proposition \ref{prop ker}, we see from Proposition \ref{prop:compare} that $ \mathcal{S} \simeq S$. Since $\mathcal{S}$ intersects trivially with $\mathcal{S}'$, we also have $\mathcal{S}'\simeq \res(\mathcal{S}')$. The claim now follows from Proposition \ref{prop:semidirect} and the above calculation.
\end{proof}

\section{A special family of finite $p$-groups of class two}\label{sec:special}

In this last section, we shall consider the finite $p$-groups of class two constructed in \cite{special_groups}. Specifically, let $n\geq 4$ and consider
\begin{align*}
 G  =\Bigg \langle x_1,x_2,\dots,x_n :&\, [[x_i,x_j],x_k]=1\mbox{ for all }1\leq i,j,k\leq n,\,\\
 &\, x_i^p = \prod_{j<k}[x_j,x_k]^{d_{i,(j,k)}}\mbox{ for all }1\leq i\leq n\Bigg\rangle.
 \end{align*}
This group has order $p^{n+{n\choose 2}}$ and has the following properties:
\begin{enumerate}[label=(\alph*)]
\item $G' = Z(G) =\mathrm{Frat}(G)$;
\item $G/G'$ is elementary abelian of order $p^{n}$ having
\begin{equation}\label{basis1}
x_1G',x_2G',\dots,x_nG'
\end{equation}
as an $\mathbb{F}_p$-basis;
\item $G'$ is elementary abelian of order $p^{n\choose 2}$ having
\begin{equation}\label{basis2}
[x_j,x_k]\mbox{ with }1\leq j < k \leq n
\end{equation}
as an $\mathbb{F}_p$-basis.
\end{enumerate}
Since $G$ has class two and $G'$ has exponent $p$, the $p$th power map
\[ \pi : G/G' \rightarrow G';\,\ (xG')^\pi = x^p\]
is a well-defined homomorphism, and the $n\times {n\choose 2}$ matrix
\[ (d_{i,(j,k)}),\mbox{ where }1\leq i\leq n\mbox{ and }1\leq j < k \leq n\]
is precisely the matrix of $\pi$ with respect to the bases (\ref{basis1}) and (\ref{basis2}).

As shown in \cite{special_groups}, we may choose $(d_{i,(j,k)})$ in a way such that 
\[\Aut(G) =\Aut_c(G),\]
so that the discussion in Section \ref{sec:group structure} applies. In particular, we have
\[ T(G) \simeq S \rtimes \res(\mathcal{S}')\]
by Theorem \ref{thm:T(G)}. For the symmetric part $S$, it is clear that
\begin{equation}\label{S}
S\simeq \mathbb{F}_p^{{n\choose 2} {n+1\choose2 }}.
\end{equation}
For the anti-symmetric part $\res(\mathcal{S}')$, there is at least a cyclic subgroup of order $p-1$ by Example \ref{example:power maps'}. This shows that $T(G)$ has a subgroup isomorphic to the semidirect product
\[\mathbb{F}_p^{{n\choose 2} {n+1\choose2 }}\rtimes \mathbb{F}_p^\times.\]
This is the precisely the content of \cite[Theorem 5.5]{class_two}, which deals with the same family of groups we are considering here. 

However, the anti-symmetric part $\res(\mathcal{S}')$, which is a subgroup of 
\[ 
 \Aut(G/G') \times \Aut(G') \simeq \GL_n(\mathbb{F}_p)\times \GL_{n\choose 2}(\mathbb{F}_p), 
\]
can potentially be much larger than $\mathbb{F}_p^\times$. We shall investigate its structure in more detail in the subsequent subsections. For now, let us just show that the anti-symmetric forms are precisely those from Example \ref{example:sigma} for the groups $G$ under consideration.
 
 \begin{prop}\label{prop:anti}We have $S' = \{\Delta_\sigma: \sigma\in \End(G')\}$.
 \end{prop}
\begin{proof}Note that we may identify $G'$ with the exterior square of $G/G'$ by associating each commutator $[x,y]$ to the wedge product $x\wedge y$. For any anti-symmetric form $\Delta\in S'$, by the universal property of the exterior square, we then see that there exists $\sigma\in \End(G')$ such that
\[ \begin{tikzcd}[column sep = 2.5cm, row sep = 2cm]
\dfrac{G}{G'}\times \dfrac{G}{G'}\arrow[swap]{d}{[\,\ ,\,\ ]} \arrow{r}{\Delta} & G'\\
G'\arrow[dotted,swap]{ru}{\sigma}
\end{tikzcd}\]
commutes, meaning that $\Delta = \Delta_\sigma$. This proves the claim since the $\Delta_\sigma$ are clearly all anti-symmetric.
\end{proof}

\subsection{The circle group of anti-symmetric forms}

Let $\sigma\in \End(G')$ and consider the $\Delta_{\sigma} \in S'$ defined in Example \ref{example:sigma}. Let 
\[ x\circ y =x\circ_\sigma y = xy\Delta_\sigma(x,y) = xy[x,y]^\sigma\]
denote the corresponding operation on $G$. We wish to know when the circle group $(G,\circ)$ is isomorphic to $G$. We may assume that 
\[1+2\sigma\in \Aut(G'),\]
for otherwise $(G,\circ)$ is not isomorphic to $G$ by Example \ref{example:sigma'}. Then
\[(G,\circ)' = G' = Z(G) = Z(G,\circ)\]
as noted in (\ref{sigma eqn}). Although $(G,\circ)$ need not be isomorphic to $G$, it still admits a presentation that is analogous to $G$, as follows.

First of all, the following are immediate from the definition of $\circ$. In fact, we see from (\ref{eqn:circ}) that (1) and (2) hold more generally, while (3) holds simply because $\Delta_\sigma$ is anti-symmetric.
\begin{enumerate}[label=(\arabic*)]
\item For all $x,y\in G$, we have $x\circ y \equiv xy \pmod{G'}$.
\item For all $x,y\in G'$, we have $x\circ y = xy$.
\item For all $x\in G$ and $k\in\mathbb{Z}$, we have $x^{\circ k} = x^{k}$, where $x^{\circ k}$ denotes the $k$th power of $x$ in the group $(G,\circ)$.
\end{enumerate}
This implies that there is no need to distinguish the operation $\circ$ from that of $G$ on the quotient $G/G'$, on the subgroup $G'$, and for arbitrary powers of elements of $G$.

Now, from (\ref{commutator}) commutators in $(G,\circ)$ are given by
\begin{equation}\label{circ comm}
[x,y]_\circ = [x,y]\Delta_\sigma(x,y)\Delta_\sigma(y,x)^{-1} = [x,y]^{1+2\sigma},\end{equation}
and we also know that $(G,\circ)$ has class two. It then follows that
\begin{align*}
 (G,\circ)  =\Bigg \langle x_1,x_2,\dots,x_n : &\, [[x_i,x_j]_\circ,x_k]_\circ=1\mbox{ for all }1\leq i,j,k\leq n,\,\\
 &\, x_i^p = \prod_{j<k}[x_j,x_k]_\circ^{d^\circ_{i,(j,k)}}\mbox{ for all }1\leq i\leq n\Bigg\rangle,
 \end{align*}
 where explicitly, the $n\times {n\choose 2}$ matrix 
\[ (d_{i,(j,k)}^\circ),\mbox{ where $1\leq i\leq n$ and $1\leq j<k\leq n$}\]
is the matrix of $\pi(1+2\sigma)^{-1}$ with respect to the bases (\ref{basis1}) and (\ref{basis2}). The relations clearly hold in $(G,\circ)$, and we have equality because the presentation on the right defines a group of the same order as $(G,\circ)$.

Since $G'$ may be identified with the exterior square of $G/G'$ via the association $[x,y]\mapsto x\wedge y$, each $\alpha\in \Aut(G/G')$ induces
\[\widehat{\alpha}\in \Aut(G')\mbox{ via } [x,y]^{\widehat{\alpha}}= [x^\alpha,y^\alpha]\mbox{ for all }x,y\in G.\]
Using this notation and (\ref{res}), we can prove:

\begin{prop}\label{prop:criterion}Let $\alpha\in \Aut(G/G')$. The following are equivalent.
\begin{enumerate}[label=\emph{(\arabic*)}]
\item
  There is an isomorphism $\theta: G\rightarrow (G,\circ)$ with $\res_c(\theta) = \alpha$.
\item The equality $\alpha^{-1}\pi\widehat{\alpha} = \pi(1+2\sigma)^{-1}$ holds. 
\end{enumerate}
Moreover, in this case, we have 
\[\res(\theta) = (\res_c(\theta),\res_z(\theta)) =  (\alpha,\widehat{\alpha}(1+2\sigma)).\] 
\end{prop}

\begin{remark}
    In the case that $\sigma=0$ is the trivial endomorphism, that is, when $(G,\circ) = G$, this  linear criterion has been
    previously employed in the literature to determine the subgroup of
    $\Aut(G/G')$ induced by $\Aut(G)$. It was introduced
    in~\cite{DH75} and was later used among others in~\cite{Car83, special_groups}.
\end{remark}

\begin{proof}Let us represent $\alpha$ as an $n\times n$ matrix
\[ (a_{ij}),\mbox{ where }1\leq i,j\leq n\]
with respect to the basis (\ref{basis1}). Letting
\[ \widehat{a}_{(j,k),(s,t)} = a_{js}a_{kt} - a_{jt}a_{ks},\]
we see that $\widehat{\alpha}$ is represented by the ${n\choose 2}\times{n\choose 2}$ matrix
\[  (\widehat{a}_{(j,k),(s,t)}),\mbox{ where }1\leq j<k\leq n\mbox{ and }1\leq s<t\leq n\]
with respect to the basis (\ref{basis2}). We may rewrite the equality in (2) as 
\begin{equation}\label{matrices} (a_{ij})(d_{i,(j,k)}^\circ) = (d_{i,(j,k)})(\widehat{a}_{(j,k),(s,t)}) \end{equation}
in terms of matrices with respect to the bases (\ref{basis1}) and (\ref{basis2}).

Let $F$ be the free group on $n$ generators $f_1,f_2,\dots,f_n$ in the variety of $p$-groups of class two and exponent $p^2$ in which all $p$th powers are central, namely the variety defined by 
\[ [[F,F],F],\,\ F^{p^2},\,\ [F^p,F].\]
By sending each $f_i$ to $x_i$, we obtain surjective homomorphisms
\[ \psi : F\rightarrow G\mbox{ and }
\psi^\circ : F\rightarrow (G,\circ).\]
We may also define a homomorphism $\widetilde{\alpha} : F\rightarrow F$ by extending
\[f_i^{\widetilde{\alpha}} = f_1^{\widetilde{a_{i1}}}f_2^{\widetilde{a_{i2}}}\cdots f_n^{\widetilde{a_{in}}}\mbox{ for all }1\leq i\leq n,\]
where $\widetilde{a_{ij}} \in \mathbb{Z}$ denotes a fixed lift of $a_{ij}\in \mathbb{F}_p$. Since $\alpha$ is invertible, the image of $\widetilde{\alpha}$ contains $f_1,f_2,\dots,f_n$ modulo $F^p$ and in particular modulo $\mathrm{Frat}(F)$. Since $\mathrm{Frat}(F)$ is contained in every maximal subgroup of $G$, this implies that $\widetilde{\alpha}$ is surjective. Since $F$ is finite, we deduce that $\widetilde{\alpha}$ is an isomorphism.

We summarize the above set-up in the following diagram:
\[ \begin{tikzcd}[column sep = 2.5cm, row sep = 1.75cm]
\ker(\psi)  \arrow[hook,swap]{r} & F\arrow[two heads]{rd}{\widetilde{\alpha}\psi^\circ}\arrow{d}{\simeq}\arrow[swap]{d}{\widetilde{\alpha}}  \arrow[two heads]{r}{\psi}& G \arrow[dotted, swap]{d}{\simeq}\arrow[dotted]{d}{\theta}\\
\ker(\psi^\circ)\arrow[hook]{r} & F\arrow[swap, two heads]{r}{\psi^\circ} & (G,\circ)
\end{tikzcd}\]
The desired map in~(1) is simply any homomorphism
$\theta: G\rightarrow (G,\circ)$ for which the above diagram
commutes. Note that $\theta$ is necessarily an isomorphism because
$\widetilde{\alpha}\psi^\circ$ is surjective. 

Now, such a homomorphism $\theta: G\rightarrow (G,\circ)$ exists if and only if 
\[ \ker(\psi)\leq \ker(\widetilde{\alpha}\psi^\circ),\]
or equivalently
\[ \begin{cases}
[[f_i,f_j],f_k]^{\widetilde{\alpha}\psi^\circ}  = 1 &\mbox{for all }1\leq i,j,k\leq n,\\
\left(f_i^{p}\right)^{\widetilde{\alpha}\psi^\circ} = \Bigg(\prod\limits_{j<k}[f_j,f_k]^{d_{i,(j,k)}}\Bigg)^{\widetilde{\alpha}\psi^\circ} &\mbox{for all $1\leq i\leq n$}.
\end{cases}\] 
The first relation is trivial because $(G,\circ)$ has class two and
\[ [[f_i,f_j],f_k]^{\widetilde{\alpha}\psi^\circ} =[[f_i^{\widetilde{\alpha}},f_j^{\widetilde{\alpha}}],f_k^{\widetilde{\alpha}}]^{\psi^\circ} = [[x_i^\alpha,x_j^\alpha]_\circ,x_k^\alpha]_\circ. \]
For the second relation, since $F$ has class two, $p$th powers are central in $F$, and $Z(F)$ has exponent $p$, the left hand side equals 
\begin{align*}
(f_i^p)^{\widetilde{\alpha}\psi^\circ} & = \left( f_1^{\widetilde{a_{i1}}}f_2^{\widetilde{a_{i2}}}\cdots f_n^{\widetilde{a_{in}}}\right)^{p\psi^\circ}\\
 & =\left( f_1^{pa_{i1}}f_2^{pa_{i2}}\cdots f_n^{pa_{in}}\right)^{\psi^\circ} \\
 & = x_1^{pa_{i1}}\circ x_2^{pa_{i2}}\circ \cdots \circ x_n^{pa_{in}}\\
 &= \prod_{\ell=1}^{n}\left( \prod_{j<k}[x_j,x_k]_\circ^{d^\circ_{\ell,(j,k)}}\right)^{a_{i\ell}}\\
 & = \prod_{j<k}[x_j,x_k]_\circ^{\sum\limits_{\ell=1}^n a_{i\ell}d^\circ_{\ell,(j,k)}},
\end{align*}
while the right hand side equals
 \begin{align*}
  \Bigg(\prod_{j<k}[f_j,f_k]^{d_{i,(j,k)}}\Bigg)^{\widetilde{\alpha}\psi^\circ} 
  & = \Bigg(\prod_{j<k}[f_j^{\widetilde{\alpha}},f_k^{\widetilde{\alpha}}]^{d_{i,(j,k)}}\Bigg)^{\psi^\circ}\\
  & = \left( \prod_{j<k} \prod_{s<t} [f_s,f_t]^{(a_{js}a_{kt} - a_{jt}a_{ks})d_{i,(j,k)}} \right)^{\psi^\circ}\\
  &= \prod_{s<t}\prod_{j<k}[x_j,x_k]_\circ^{(a_{sj}a_{tk}-a_{sk}a_{tj})d_{i,(s,t)}}\\
  & = \prod_{j<k}[x_j,x_k]_\circ^{\sum\limits_{s<t} d_{i,(s,t)}\widehat{a}_{(s,t),(j,k)}}.
 \end{align*}
By (\ref{basis2}) and (\ref{circ comm}), we know that
\[ [x_j,x_k]_\circ,\mbox{ where }1\leq j < k \leq n\]
is also an $\mathbb{F}_p$-basis for $G'$. We then see that the second relation, when $1\leq i\leq n$ is fixed, holds if and only if
 \[ \sum_{\ell=1}^n a_{i\ell}d_{\ell,(j,k)}^\circ = \sum_{s<t} d_{i,(s,t)}\widehat{a}_{(s,t),(j,k)}\mbox{ for all $1\leq j < k\leq n$}.\]
But the sum on the left is the $i,(j,k)$ entry of the matrix on the left of (\ref{matrices}), while that on the right is the $i,(j,k)$ entry of the matrix on the right of (\ref{matrices}).

We have thus shown that the existence of the $\theta : G\rightarrow (G,\circ)$ in (1) is indeed equivalent to the equality in (2). Moreover, in this case, the action of $\theta$ on $G'$ is given by
\[ [x,y]^\theta = [x^\theta,y^\theta]_\circ = [x^\alpha,y^\alpha]^{1+2\sigma} = [x,y]^{\widehat{\alpha}(1+2\sigma)},\]
and hence $\res_z(\theta) = \widehat{\alpha}(1+2\sigma)$, as desired.
\end{proof}

\subsection{Proof of Theorem \ref{thm:main}}

By Propositions \ref{prop:anti} and \ref{prop:criterion}, we have
\begin{align*}
 \res(\mathcal{S}') &= \{ (\alpha,\widehat{\alpha}(1+2\sigma)): \alpha\in\Aut(G/G')\mbox{ and }\sigma\in \End(G')\mbox{ such }\\ &\hspace{2cm}\mbox{that }1+2\sigma\in \Aut(G')\mbox{ and }\alpha^{-1}\pi\widehat{\alpha} = \pi(1+2\sigma)^{-1} \}.
 \end{align*}
Making a change of variables $\tau = 1+2\sigma$, we then obtain
\begin{align}\label{res description}
 \res(\mathcal{S}') &= \{ (\alpha,\widehat{\alpha}\tau): \alpha\in\Aut(G/G')\mbox{ and }\tau\in \Aut(G')\\\notag &\hspace{2cm}\mbox{which satisfy the relation }\alpha^{-1}\pi\widehat{\alpha} = \pi \tau^{-1} \}.
 \end{align}
Recall that $\pi : G/G'\rightarrow G'$ denotes the $p$th power map, and the structure of $\res(\mathcal{S}')$ depends upon $\pi$ by the above description. Since $G$ has class two, the so-called omega subgroup of $G$ equals the set
\[ \Omega_1(G) = \{x\in G : x^p = 1\}\]
consisting of the elements of order dividing $p$. Note that
\begin{align}\label{iff}
 \Omega_1(G)\mbox{ is contained in }G'&\iff\pi \mbox{ is injective} \\\notag
 &\iff (d_{i,(j,k)})\mbox{ has full rank}.\end{align}
 Recall that $(d_{i,(j,k)})$ is the matrix of $\pi$ with respect to the bases (\ref{basis1}) and (\ref{basis2}). The second equivalence here holds because our matrices act on row vectors from the right. In this case, we can prove:

\begin{prop}\label{prop:full rank}Assume that $\Omega_1(G)$ is contained in $G'$. 

We have an isomorphism
\[ \res(\mathcal{S}') \simeq \mathbb{F}_p^{(n'-n)\times n}\rtimes \left( \GL_n(\mathbb{F}_p)\times \GL_{n'-n}(\mathbb{F}_p)\right),\]
where $n' = {n\choose 2}$ and the semidirect product action is the natural one as given in the statement of Theorem \ref{thm:main}.
\end{prop}
 
 \begin{proof}The assumption implies that $\pi$ is represented by
 \[ D = \begin{bmatrix} 1  &&&& 0 &\cdots & 0\\
 & 1 && &0 & \cdots & \\
 &&\ddots && \vdots & \vdots & \vdots\\
 &&&1 & 0 & \cdots & 0\end{bmatrix}\]
with respect to a suitable choice of bases of $G/G'$ and $G'$. In terms of matrices, we then see from (\ref{res description}) that
 \begin{align*}
  \res(\mathcal{S}') & \simeq \{ (A,\widehat{A}T): A \in \GL_n(\mathbb{F}_p)\mbox{ and }T\in \GL_{n'}(\mathbb{F}_p)\\
  &\hspace{2cm}\mbox{which satisfy the relation }A^{-1}D\widehat{A} = DT^{-1}\}.
 \end{align*}
 Here, with respect to the chosen bases, if $A$ is the matrix representing $\alpha\in \Aut(G/G')$ then $\widehat{A}$ is the matrix representing $\widehat{\alpha}\in \Aut(G')$, while $T$ is the matrix representing $\tau$. 
 
Let us fix an arbitrary $A\in \GL_n(\mathbb{F}_p)$, and we wish to determine all the solutions $T \in \GL_{n'}(\mathbb{F}_p)$ to the equation
\begin{equation}\label{matrix eqn}
A^{-1}D\widehat{A} = DT^{-1},\mbox{ or equivalently }A^{-1}D=DT^{-1}\widehat{A}^{-1}.
\end{equation}
By the shape of the matrix $D$, the left hand side is the block matrix
\[ \begin{bmatrix} A^{-1} & 0\end{bmatrix}.\]
Also, the equation (\ref{matrix eqn}) only imposes restrictions on and uniquely determines the first $n$ rows of $T^{-1}\widehat{A}^{-1}$. In particular, we see that
\[ T^{-1}\widehat{A}^{-1} = \begin{bmatrix}
A^{-1} & 0 \\ 
-Q & M^{-1}
\end{bmatrix},\]
where $Q$ is arbitrary but $M$ has to be invertible because we require $T$ to be invertible. We now conclude that the solutions $T\in \GL_{n'}(\mathbb{F}_p)$ to the equation (\ref{matrix eqn}) are precisely
\[ T = \widehat{A}^{-1}\begin{bmatrix}
A^{-1} & 0 \\ 
-Q & M^{-1}
\end{bmatrix}^{-1} = \widehat{A}^{-1}\begin{bmatrix} 
A & 0\\
MQA & M
\end{bmatrix},\]
where $Q\in \mathbb{F}_p^{(n'-n)\times n}$ and $M\in \GL_{n'-n}(\mathbb{F}_p)$ are arbitrary.

The above shows that elements of $\res(\mathcal{S}')$ may be parametrized by
\[  \mathbb{F}_p^{(n'-n)\times n}\times \GL_n(\mathbb{F}_p)\times \GL_{n'-n}(\mathbb{F}_p).\]
Let us now compute the structure of $\res(\mathcal{S}')$. For any
\[Q_1,Q_2\in \mathbb{F}_p^{(n'-n)\times n},\, A_1,A_2\in \GL_n(\mathbb{F}_p),\, M_1,M_2\in \GL_{n'-n}(\mathbb{F}_p),\]
letting $Q = M_2^{-1}Q_1 + Q_2A_1^{-1}$, we have
\begin{align*}
&\hspace{5.25mm}
\left(A_1,\begin{bmatrix} 
A_1 & 0\\
M_1Q_1A_1 & M_1
\end{bmatrix}\right)\left(A_2,\begin{bmatrix} 
A_2 & 0\\
M_2Q_2A_2 & M_2
\end{bmatrix}\right)\\
& = \left(A_1A_2,\begin{bmatrix}
A_1A_2 & 0 \\
M_1M_2 QA_1A_2 & M_1M_2
\end{bmatrix} \right).
\end{align*}
This means that in terms of the above parametrization, the multiplication of $\res(\mathcal{S}')$ is given by
\[ (Q_1,A_1,M_1)\cdot (Q_2,A_2,M_2) = (M_2^{-1}Q_1 + Q_2A_1^{-1}, A_1A_2,M_1M_2).\]
From here, we see that the structure of $\res(\mathcal{S}')$ is as claimed.
 \end{proof}

One can now appeal to the construction
in~\cite{special_groups} to obtain that the condition $\Aut(G) =\Aut_c(G)$ can always be realized by a choice of a matrix $(d_{i,(j,k)})$ which has full rank, so that $\Omega_1(G)\leq G'$ as noted in (\ref{iff}). Theorem \ref{thm:main} is now an immediate consequence of (\ref{S}), Theorem \ref{thm:T(G)}, and Proposition \ref{prop:full rank}.

\bibliographystyle{amsalpha}
 
\bibliography{Refs}

\providecommand{\bysame}{\leavevmode\hbox to3em{\hrulefill}\thinspace}
\providecommand{\MR}{\relax\ifhmode\unskip\space\fi MR }
\providecommand{\MRhref}[2]{%
  \href{http://www.ams.org/mathscinet-getitem?mr=#1}{#2}
}
\providecommand{\href}[2]{#2}
\begin{thebibliography}{CDV18}

\bibitem[Car83]{Car83}
A.~Caranti, \emph{Automorphism groups of {$p$}-groups of class {$2$} and
  exponent {$p^{2}$}: a classification on {$4$} generators}, Ann. Mat. Pura
  Appl. (4) \textbf{134} (1983), 93--146. \MR{736737}

\bibitem[Car16]{special_groups}
\bysame, \emph{A simple construction for a class of {$p$}-groups with all of
  their automorphisms central}, Rend. Semin. Mat. Univ. Padova \textbf{135}
  (2016), 251--258. \MR{3506071}

\bibitem[Car18]{class_two}
\bysame, \emph{Multiple holomorphs of finite {$p$}-groups of class two}, J.
  Algebra \textbf{516} (2018), 352--372. \MR{3863482}

\bibitem[Car20]{mybi}
\bysame, \emph{Bi-skew braces and regular subgroups of the holomorph}, J.
  Algebra \textbf{562} (2020), 647--665. \MR{4130907}

\bibitem[CDV17]{fgabelian}
A.~Caranti and F.~Dalla~Volta, \emph{The multiple holomorph of a finitely
  generated abelian group}, J. Algebra \textbf{481} (2017), 327--347.
  \MR{3639478}

\bibitem[CDV18]{perfect}
\bysame, \emph{Groups that have the same holomorph as finite perfect group}, J.
  Algebra \textbf{507} (2018), 81--102. \MR{3807043}

\bibitem[Chi19]{Childs-bi-skew}
Lindsay~N. Childs, \emph{Bi-skew braces and {H}opf {G}alois structures}, New
  York J. Math. \textbf{25} (2019), 574--588. \MR{3982254}

\bibitem[CS22]{CS2}
A.~Caranti and L.~Stefanello, \emph{Brace blocks from bilinear maps and
  liftings of endomorphisms}, J. Algebra \textbf{610} (2022), no.~15, 831--851.
  \MR{4473766}

\bibitem[DH75]{DH75}
G.~Daues and H.~Heineken, \emph{Dualit\"{a}ten und {G}ruppen der {O}rdnung
  {$p^6$}}, Geometriae Dedicata \textbf{4} (1975), no.~2, /3/4, 215--220.
  \MR{401907}

\bibitem[Koc22]{Koc22}
Alan Koch, \emph{Abelian maps, brace blocks, and solutions to the
  {Y}ang-{B}axter equation}, J. Pure Appl. Algebra \textbf{226} (2022), no.~9,
  Paper No. 107047. \MR{4381676}

\bibitem[Koh15]{Kohl}
Timothy Kohl, \emph{Multiple holomorphs of dihedral and quaternionic groups},
  Comm. Algebra \textbf{43} (2015), no.~10, 4290--4304. \MR{3366576}

\bibitem[Mil08]{Miller}
G.~A. Miller, \emph{On the multiple holomorphs of a group}, Math. Ann.
  \textbf{66} (1908), no.~1, 133--142. \MR{1511494}

\bibitem[Mil51]{Mills}
W.~H. Mills, \emph{Multiple holomorphs of finitely generated abelian groups},
  Trans. Amer. Math. Soc. \textbf{71} (1951), 379--392. \MR{45117}

\bibitem[Tsa19]{ASG}
Cindy Tsang, \emph{On the multiple holomorph of a finite almost simple group},
  New York J. Math. \textbf{25} (2019), 949--963. \MR{4012575}

\bibitem[Tsa20]{squarefree}
\bysame, \emph{On the multiple holomorph of groups of squarefree or odd prime
  power order}, J. Algebra \textbf{544} (2020), 1--28. \MR{4023139}

\bibitem[Tsa21]{arXiv:2107.13690}
\bysame, \emph{The multiple holomorph of centerless groups}, arXiv, 2021,
  \url{https://arxiv.org/abs/2107.13690}.

\bibitem[Tsa22]{metacyclic}
\bysame, \emph{The multiple holomorph of split metacyclic $p$-groups}, Comm.
  Algebra \textbf{50} (2022), no.~10, 4269--4287. \MR{4447460}

\end{thebibliography}

\end{document}